\documentclass[12pt,leqno]{article}

\usepackage[lite]{amsrefs}
\usepackage{latexsym,enumerate}

 \usepackage[notref, notcite]{}
\usepackage{amssymb, amsfonts, eucal,  amsmath,amsthm}

\usepackage{color}
\definecolor{added}{rgb}{0, 0, 1}
\definecolor{deleted}{rgb}{1, 0, 0}

\newtheorem{theorem}{Theorem}[section]

\newtheorem{lemma}[theorem]{Lemma}

\newtheorem{remark}[theorem]{Remark}

\newtheorem{corollary}[theorem]{Corollary}

\newcommand{\sect}[1]{\section{#1} \setcounter{equation}{0} }

\newcounter{ca}
\newcommand{\norm}[2]{\left\|#1\right\|_{#2}}

\newcommand{\Pn}{\mathbb P_n}

 \newcommand{\ec}{\end{comment}}
\newcommand{\bc}{ \begin{comment}
 }

\newcommand{\calS}{{\mathcal S}}

\newcommand{\I}{{\mathcal I}}

\newcommand\w{{\omega}}
\def\be  {\begin{equation}}
\def\ee  {\end{equation}}
\def\ba  {\begin{eqnarray}}
\def\ea  {\end{eqnarray}}
\def\baa {\begin{eqnarray*}}
\def\eaa {\end{eqnarray*}}
\newenvironment{comment}[2]
{\bgroup\vspace{7pt}
\begin{tabular}{|p{5in}|}
\hline \qquad \bf \footnotesize Comment -- to be deleted in the final version \\
\hline
\quad\sl\footnotesize #1#2} {\\ \hline \end{tabular}
\vspace{7pt}\indent\egroup}

\def\updots{\mathinner{\mkern
1mu\raise 1pt \hbox{.}\mkern 2mu \mkern 2mu \raise
4pt\hbox{.}\mkern 1mu \raise 7pt\vbox {\kern 7 pt\hbox{.}}} }

\newcommand{\B}{\mathbb B}

\newcommand{\N}{\mathbb N}

\renewcommand{\a}{\alpha}
\renewcommand{\b}{\beta}
\newcommand{\ineq}[1]{{\rm(\ref{#1})}}

\newcommand{\ie}{{\em i.e., }}
\newcommand{\eg}{{\em e.g. }}

\newcommand{\bpic}{
\begin{center}
}

\newcommand{\epic}{
\endpspicture
\end{center}
}

\renewcommand{\L}{L}
\newcommand{\Lp}{\L_p}

\newcommand{\Poly}{\mathbb P}

\newcommand{\wkr}{\w_{k,r}^\varphi}
\newcommand{\okr}{\Omega_{k,r}^\varphi}
\newcommand{\wokr}{\widetilde{\Omega}_{k,r}^\varphi}

\newcommand{\bwkr}{\Psi_{k,r}^\varphi}
 \newcommand{\Dom}{{\mathfrak{D}}}

\newcommand{\ddelta}{\mu(\delta)}

 \newcommand{\wt}{{\mathcal{W}}}

\newcommand{\ccc}{\boldsymbol \varrho }

\newcommand{\thm}[1]{Theorem~\ref{#1}}
\newcommand{\lem}[1]{Lemma~\ref{#1}}
\newcommand{\cor}[1]{Corollary~\ref{#1}}

 \newcommand{\Lpab}{\L_p^{\alpha,\beta}}
 \newcommand{\wab}{w_{\a,\b}}
 \newcommand{\weight}{\wt_{kh}^{r/2+\a,r/2+\b}}

\title{{\sc On some properties of moduli of smoothness with Jacobi weights}
\thanks{{\it AMS classification:} 41A10, 41A17, 41A25. {\it Keywords
and phrases:} Approximation by polynomials in weighted $L_p$-norms, Jacobi weights, moduli of smoothness. }}

\author{K. A.  Kopotun\thanks{Department of Mathematics, University of
Manitoba, Winnipeg, Manitoba, R3T 2N2, Canada ({\tt
kirill.kopotun@umanitoba.ca}). Supported by NSERC of Canada Discovery Grant RGPIN \mbox{04215-15.}} ,
D. Leviatan\thanks{Raymond and Beverly Sackler School of Mathematical
Sciences, Tel Aviv University, Tel Aviv 6139001, Israel ({\tt
leviatan@post.tau.ac.il}).}\ \ and I. A. Shevchuk\thanks
{Faculty of Mechanics and Mathematics, National Taras
Shevchenko University of Kyiv, 01033 Kyiv, Ukraine ({\tt
shevchuk@univ.kiev.ua}).} }

\begin{document}

\maketitle

\centerline{\sl\small Dedicated to the memory of our friend, colleague and collaborator}
\centerline{\sl\small Yingkang Hu (July 6, 1949 -- March 11, 2016)}

\abstract{
We discuss some properties of the moduli of smoothness with Jacobi weights that we have recently introduced and that are defined as
\[
\wkr(f^{(r)},t)_{\a,\b,p}   :=\sup_{0\leq h\leq t}\norm{ \weight(\cdot)
\Delta_{h\varphi(\cdot)}^k (f^{(r)},\cdot)}p
\]
where $\varphi(x) = \sqrt{1-x^2}$, $\Delta_h^k(f,x)$ is the $k$th symmetric difference of $f$ on $[-1,1]$,
\[
\wt_\delta^{\xi,\zeta} (x):=  (1-x-\delta\varphi(x)/2)^\xi
(1+x-\delta\varphi(x)/2)^\zeta ,
\]
and $\a,\b > -1/p$ if $0<p<\infty$, and $\a,\b \geq 0$ if $p=\infty$.

We show, among other things, that for all $m, n\in \N$, $0<p\le \infty$,    polynomials $P_n$ of degree $<n$ and sufficiently small $t$,
\begin{align*}
 \w_{m,0}^\varphi(P_n, t)_{\a,\b,p} & \sim t \w_{m-1,1}^\varphi(P_n', t)_{\a,\b,p} \sim \dots \sim t^{m-1}\w_{1,m-1}^\varphi(P_n^{(m-1)}, t)_{\a,\b,p} \\
 & \sim t^m \norm{\wab \varphi^{m} P_n^{(m)}}{p} ,
\end{align*}
where $\wab(x)  = (1-x)^\a (1+x)^\b$ is the usual Jacobi weight.

In the spirit of Yingkang Hu's work, we   apply this to characterize the behavior of the polynomials of best approximation of a function in a Jacobi weighted $L_p$ space, $0<p\le\infty$.   Finally we discuss sharp Marchaud and Jackson type inequalities in the case $1<p<\infty$.

 }

\sect{Introduction}

Recall that the Jacobi weights are defined as $\wab(x):=(1-x)^\a(1+x)^\b $, where parameters $\a$ and $\b$ are usually assumed to be such   that $\wab \in \Lp[-1,1]$, \ie
\[
 \a,\b \in
J_p
 := \begin{cases}
(-1/p, \infty), & \mbox{\rm if } 0<p<\infty , \\
[0,\infty), & \mbox{\rm if } p=\infty .
\end{cases}
\]
We   denote by $\Pn$ the set of all algebraic polynomials of degree  $\le n-1$, and
$\Lpab (I):=\left\{f\mid \norm{\wab f}{\Lp(I)}<\infty \right\}$, where  $I\subseteq [-1,1]$. For convenience, if $I=[-1,1]$ then we omit $I$ from the notation. For example,
$\norm{\cdot}{p}:= \norm{\cdot}{\Lp[-1,1]}$, $\Lpab := \Lpab [-1,1]$,  etc.

Following \cite{sam} we denote $\B^0_p(\wab) := \Lpab$, and
\[
\B_p^r(\wab):=\left\{\,f\,|\,f^{(r-1)}\in AC_{loc} \quad\text{and}\quad\varphi^rf^{(r)}\in \Lpab \right\}, \quad r\ge 1,
\]
where $AC_{loc}$ denotes the set of  functions which are locally absolutely continuous in $(-1,1)$, and $\varphi(x):= \sqrt{1-x^2}$.
Also (see \cite{sam}), for $k,r\in\N$ and
$f\in \B^r_p(\wab)$,  let
\begin{align} \label{wkrdefinition}
\wkr(f^{(r)},t)_{\a,\b,p} & :=\sup_{0\leq h\leq t}\norm{ \weight(\cdot)
\Delta_{h\varphi(\cdot)}^k (f^{(r)},\cdot)}p \\ \nonumber
& =
\sup_{0<h\leq t}\norm{\weight(\cdot)
\Delta_{h\varphi(\cdot)}^k(f^{(r)},\cdot)}{L_p(\Dom_{kh})} ,
\end{align}
where
\[
\Delta_h^k(f,x; S):=\left\{
\begin{array}{ll}
\sum_{i=0}^k  \binom{k}{i}
(-1)^{k-i} f(x-\frac{kh}{2}+ih),&\mbox{\rm if }\, [x-\frac{kh}{2}, x+\frac{kh}{2}]  \subseteq S \,,\\
0,&\mbox{\rm otherwise},
\end{array}\right.
\]
is the $k$th symmetric difference,  $\Delta_h^k(f,x) := \Delta_h^k(f,x; [-1,1])$,
\[
\wt_\delta^{\xi,\zeta} (x):=  (1-x-\delta\varphi(x)/2)^\xi
(1+x-\delta\varphi(x)/2)^\zeta,
\]
and
\begin{align*}
\Dom_\delta &:= 
 [-1+\ddelta,1-\ddelta], \quad \ddelta:=2\delta^2/(4+\delta^2)
\end{align*}
(note that ${\Delta}_{h\varphi(x)}^k(f,x)=0$   if $x\not\in\Dom_{kh}$).

We define the main part weighted modulus of smoothness as
\be \label{mainpart}
\okr(f^{(r)}, A, t)_{\a, \b, p} :=
\sup_{0\leq h\leq t}\norm{ \wab(\cdot) \varphi^r(\cdot)
\Delta_{h\varphi(\cdot)}^k (f^{(r)},\cdot; \I_{A,h})}{\Lp(\I_{A,h})} ,
\ee
 where $\I_{A,h} := [-1+Ah^2, 1-Ah^2]$  and $A > 0$.

We also denote
\begin{align} \label{bwkrdefinition}
\bwkr(f^{(r)},t)_{\a,\b,p} 
& :=\sup_{0\leq h\leq t}\norm{ \wab(\cdot) \varphi^r(\cdot)
\Delta_{h\varphi(\cdot)}^k (f^{(r)},\cdot)}p , 
\end{align}
\ie   $\bwkr$ is ``the main part modulus $\okr$ with $A=0$''. However, we want to emphasize that while $\okr(f^{(r)}, A, t)_{\a, \b, p}$ with $A>0$  and $\wkr(f^{(r)},t)_{\a,\b,p}$ are bounded for all $f\in \B^r_p(\wab)$ (see \cite{sam}*{Lemma 2.4}), modulus $\bwkr(f^{(r)},t)_{\a,\b,p}$ may be infinite for such functions  (for example, this is the case for $f$  such that $f^{(r)}(x) = (1-x)^{-\gamma}$ with  $1/p\le \gamma < \a+r/2+1/p$).

\begin{remark}
We note that the main part modulus is sometimes defined with the difference inside the norm not restricted to $\I_{A,h}$, \ie
\be \label{dtmainpart}
  \wokr(f^{(r)}, A, t)_{\a, \b, p} :=
\sup_{0\leq h\leq t}\norm{ \wab(\cdot) \varphi^r(\cdot)
\Delta_{h\varphi(\cdot)}^k (f^{(r)},\cdot)}{\Lp(\I_{A,h})} .
\ee
Clearly, $\okr(f^{(r)}, A, t)_{\a, \b, p} \le \wokr(f^{(r)}, A, t)_{\a, \b, p}$.
Moreover,   we have an estimate  in the opposite direction as well if we replace $A$ with a larger constant $A'$. For example,
$\wokr(f^{(r)}, A', t)_{\a, \b, p} \le \okr(f^{(r)}, A, t)_{\a, \b, p}$, where $A' = 2\max\{A, k^2\}$ $($see \ineq{tmp1}$)$. At the same time, if $A$ is so small that $\Dom_{kh} \subset \I_{A,h}$ (for example, if $A\le k^2/4$), then
 $\wokr(f^{(r)}, A, t)_{\a, \b, p} = \bwkr(f^{(r)},t)_{\a,\b,p}$. Hence, all our results in this paper are valid with the modulus \ineq{mainpart} replaced by \ineq{dtmainpart} with an additional assumption that $A$ is sufficiently large $($assuming that $A\ge 2k^2$ will do$)$.
\end{remark}


Throughout this paper, we use the notation
\[
q := \min\{1, p\} ,
\]
and   $\ccc$ stands for some sufficiently small positive constant depending only on
$\a$, $\b$, $k$ and $q$, and independent of $n$, to be prescribed in the proof of Theorem \ref{mainthm}.

\sect{The main result}

The following theorem is our main result.

\begin{theorem} \label{mainthm}
Let $k, n\in \N$, $r\in\N_0$, $A> 0$,  $0<p\le \infty$,   $\a+r/2, \b+r/2 \in  J_p$, and let  $0<t \leq \ccc  n^{-1}$, where $\ccc$ is some positive constant that depends only on $\a$, $\b$, $k$ and $q$.
 Then, for any $P_n\in\Pn$,
\begin{align}  \label{main}
\wkr(P_n^{(r)}, t)_{\a,\b,p} & \sim   \bwkr(P_n^{(r)}, t)_{\a,\b,p}   \sim \okr(P_n^{(r)}, A, t)_{\a, \b, p}\\ \nonumber
& \sim  t^{k} \norm{\wab \varphi^{k+r} P_n^{(k+r)}}{p} ,
\end{align}
where the equivalence constants depend only on $k$, $r$, $\a$, $\b$, $A$ and $q$.
\end{theorem}

The following is an immediate corollary of \thm{mainthm} by virtue of the fact that, if $\a,\b \in J_p$, then $\a+r/2, \b+r/2 \in  J_p$ for all $r\ge 0$.

\begin{corollary} \label{maincor}
Let $m, n\in \N$, $A> 0$,  $0<p\le \infty$,  $\a, \b \in  J_p$,   and let  $0<t \leq \ccc  n^{-1}$. Then, for any $P_n\in\Pn$, and any $k\in\N$ and $r\in \N_0$ such that $k+r=m$,
\begin{align*}
 t^{-k} \wkr(P_n^{(r)}, t)_{\a,\b,p} & \sim  t^{-k} \bwkr(P_n^{(r)}, t)_{\a,\b,p}   \sim t^{-k} \okr(P_n^{(r)}, A, t)_{\a, \b, p}\\
& \sim   \norm{\wab \varphi^{m} P_n^{(m)}}{p} ,
\end{align*}
where the equivalence constants depend only on $m$, $\a$, $\b$, $A$ and $q$.
\end{corollary}

It was shown in \cite{sam}*{Corollary 1.9} that, for $k\in\N$, $r\in \N_0$, $r/2+\a \ge 0$, $r/2+\b \ge 0$, $1\le p \le \infty$, $f\in \B_p^r(\wab)$,  $\lambda \ge 1$ and all $t>0$,
\[
\wkr(f^{(r)}, \lambda t)_{\a,\b,p} \le c \lambda^k \wkr(f^{(r)},  t)_{\a,\b,p} .
\]
Hence, in the case $1\le p \le \infty$, we can strengthen \cor{maincor} for the moduli $\wkr$. Namely,
 the following result is valid.

\begin{corollary} \label{anothercor}
Let $m, n\in \N$,    $1\le p\le \infty$,  $\a, \b \in  J_p$, $\Lambda >0$  and let  $0<t \leq  \Lambda  n^{-1}$. Then, for any $P_n\in\Pn$, and any $k\in\N$ and $r\in \N_0$ such that $k+r=m$,
\[
 t^{-k} \wkr(P_n^{(r)}, t)_{\a,\b,p}
  \sim   \norm{\wab \varphi^{m} P_n^{(m)}}{p} ,
\]
where the equivalence constants depend only on $m$, $\a$, $\b$ and $\Lambda$.
\end{corollary}

 \begin{remark}\label{rem} In the case $1\le p \le \infty$, several equivalences in \thm{mainthm} and \cor{maincor}   follow  from \cite{hl}*{Theorems 4 and 5}, since, as was shown in \cite{sam}*{(1.8)}, for $1\le p \le \infty$,
 \be \label{equival}
\wkr(f^{(r)},t)_{\a,\b,p}\sim\omega^k_\varphi( f^{(r)},t)_{w_{\a,\b}\varphi^r,p}, \quad 0<t\leq t_0 ,
\ee
where $\omega^k_\varphi( g,t)_{w,p}$ is the three-part weighted Ditzian-Totik modulus of smoothness (see \eg \cite{sam}*{(5.1)} for its definition).

Note that it is still an open problem if \ineq{equival} is valid if $0<p<1$.
\end{remark}

\begin{proof}[Proof of $\thm{mainthm}$]
The main idea of the proof is not much different from that of \cite{hl}*{Theorems 3-5}.

First, we note that it suffices to prove \thm{mainthm} in the case $r=0$. Indeed, suppose we proved that, for $k,n\in\N$, $A>0$, $0<t \le\ccc n^{-1}$, $0<p\le \infty$, $\a,\b\in J_p$ and any polynomial $Q_n\in\Pn$,
\begin{align}  \label{r0}
 \w_{k,0}^\varphi(Q_n , t)_{\a,\b,p}   & \sim   \Psi_{k,0}^\varphi(Q_n , t)_{\a,\b,p}   \sim \Omega_{k,0}^\varphi (Q_n , A, t)_{\a, \b, p} \\ \nonumber
  &\sim   t^{k } \norm{\wab \varphi^{k } Q_n^{(k )}}{p}.
\end{align}
Then, if $P_n$ is an arbitrary polynomial from $\Pn$, and   $r$ is an arbitrary natural number,
assuming that $n > r$ (otherwise, $P_n^{(r)} \equiv 0$ and there is nothing to prove) and denoting $Q := P_n^{(r)} \in \Poly_{n-r}$, we have
\[
\wkr(P_n^{(r)}, t)_{\a,\b, p} = \w_{k,0}^\varphi(Q, t)_{\a+r/2,\b+r/2, p} ,
\]
\[
\bwkr(P_n^{(r)}, t)_{\a,\b, p} = \Psi_{k,0}^\varphi(Q, t)_{\a+r/2,\b+r/2, p} ,
\]
\[
\okr(P_n^{(r)}, t)_{\a,\b, p} = \Omega_{k,0}^\varphi(Q, A, t)_{\a+r/2,\b+r/2, p}
\]
and
\[
  \norm{\wab \varphi^{k+r} P_n^{(k+r)}}{p} =  \norm{\w_{\a+r/2, \b+r/2} \varphi^k  Q^{(k)}}{p} ,
\]
and so \ineq{main} follows from \ineq{r0} with $\a$ and $\b$ replaced by $\a+r/2$ and $\b+r/2$, respectively.

Now, note that it immediately follows from the definition that
\[
\w_{k,0}^\varphi(g , t)_{\a,\b,p} \le \Psi_{k,0}^\varphi(g,  t)_{\a, \b, p} .
\]
Also, for $A > 0$,
\[
\Omega_{k,0}^\varphi(g, A, t)_{\a, \b, p}  \le c \w_{k,0}^\varphi(g,  t)_{\a, \b, p} ,
\]
since $\wab(x) \le c \wt_{kh}^{\a,\b}(x)$ for $x$ such that $x\pm kh\varphi(x)/2 \in \I_{A,h}$.

Hence, in order to prove \ineq{r0}, it suffices to show that
\be \label{1}
\Psi_{k,0}^\varphi (Q_n, t)_{\a,\b,p}  \le c t^{k} \norm{\wab \varphi^{k} Q_n^{(k)}}{p}
\ee
and
\be \label{2}
t^{k} \norm{\wab \varphi^{k} Q_n^{(k)}}{p}  \le c \Omega_{k,0}^\varphi(Q_n , A, t)_{\a, \b, p}  .
\ee

Recall the following Bernstein-Dzyadyk-type inequality that follows from \cite{hl}*{(2.24)}: if $0<p\le \infty$,  $\a,\b \in J_p$ and $P_n\in\Pn$, then
\[
\norm{\wab \varphi^s P_n'}{p} \leq c ns \norm{\wab \varphi^{s-1} P_n}{p},   \quad 1\le s \le n-1 ,
\]
where $c$ depends only on $\a$,  $\b$ and $q$,
and is independent of $n$ and $s$.

This implies that, for any $Q_n \in \Pn$ and $k,j \in \N$,
\be \label{corbern}
\norm{\wab \varphi^{k+j}  Q_n^{(k+j)}}{p} \le (c_0  n)^j \frac{ (k+j)!}{k!}   \norm{\wab \varphi^{k}  Q_n^{(k)}}{p}, \quad 1\le k+j \le n-1 .
\ee


We now use the following identity (see \cite{hl}*{(2.4)}):
\begin{quote}
for any $Q_n\in\Pn$ and $k\in\N$, we have
\be \label{ident}
\Delta^k_{h\varphi(x)}(Q_n, x) = \sum_{i=0}^{K} \frac{1}{(2i)!}    \varphi^{k+2i}(x) Q_n^{(k+2i)}(x) h^{k+2i} \xi_{k+2i}^{2i} ,
\ee
where $K := \lfloor (n-1-k)/2 \rfloor$, and
 $\xi_j \in (-k/2, k/2)$ depends only on $k$ and $j$.
\end{quote}
Applying \ineq{corbern}, we obtain, for $0\le i \le K$ and $0<h\le t \le \ccc n^{-1}$,
\begin{align*}
\norm{\frac{1}{(2i)!} \wab \varphi^{k+2i}  Q_n^{(k+2i)}  }{p} h^{2i} |\xi_{k+2i}|^{2i} & \le   (c_0 \ccc k/2)^{2i} \frac{(k+2i)!}{(2i)! k!}  \norm{\wab \varphi^{k}  Q_n^{(k)}}{p} \\
& \le [c_0 \ccc k(k+1)/2]^{2i}    \norm{\wab \varphi^{k}  Q_n^{(k)}}{p} \\
& \le B^{2i} \norm{\wab \varphi^{k}  Q_n^{(k)}}{p} ,
\end{align*}
where we used the estimate $(k+2i)!/((2i)! k!) \le (k+1)^{2i}$, and where $\ccc$ is taken so small that the last estimate holds with $B:= (1/3)^{1/(2q)}$. Note that
 $\sum_{i=1}^\infty B^{2iq} = 1/2$.

Hence, it follows from \ineq{ident} that
\begin{align*}
\norm{\wab   \Delta^k_{h\varphi}(Q_n, \cdot)}{p}^q
& \le  h^{kq}
\sum_{i=0}^{K}  \norm{\frac{1}{(2i)!} \wab  \varphi^{k+2i}  Q_n^{(k+2i)} }{p}^q  h^{2iq} |\xi|_{k+2i}^{2iq} \\
& \le h^{kq} \norm{\wab \varphi^{k}  Q_n^{(k)}}{p}^q
\left( 1+ \sum_{i=1}^{K} B^{2iq} \right) \\
& \le
3/2 \cdot h^{kq} \norm{\wab \varphi^{k}  Q_n^{(k)}}{p}^q .
\end{align*}
This immediately implies
\[
\Psi_{k,0}^\varphi (Q_n, t)_{\a,\b,p}  \le (3/2)^{1/q}  t^{k} \norm{\wab \varphi^{k} Q_n^{(k)}}{p} ,
\]
and so \ineq{1} is proved.

Recall now the following Remez-type inequality (see \eg \cite{hl}*{(2.22)}):
\begin{quote}
If  $0<p\le \infty$, $\a, \b\in J_p$, $a \ge 0$, $n\in\N$ is such that $n> \sqrt{a}$, and $P_n\in\Pn$, then
\be \label{remez}
\norm{\wab P_n}{p} \leq c \norm{\wab P_n}{\Lp[-1+a n^{-2}, 1-an^{-2}]} ,
\ee
where $c$ depends only on $\a$, $\b$, $a$ and $q$.
\end{quote}
Note that
\begin{align*}
\Omega_{k,0}^\varphi (Q_n, A, t)_{\a, \b, p}  & =
\sup_{0\leq h\leq t}\norm{ \wab(\cdot)
\Delta_{h\varphi(\cdot)}^k (Q_n,\cdot; \I_{A,h})}{\Lp(\I_{A,h})}  \\
& =
\sup_{0\leq h\leq t}\norm{ \wab(\cdot)
\Delta_{h\varphi(\cdot)}^k (Q_n,\cdot)}{\Lp(\calS_{k,A,h})} ,
\end{align*}
 where the set $\calS_{k,A,h}$ is an interval containing all $x$ so that $x \pm kh\varphi(x)/2 \in \I_{A,h}$. Observe that
 \[
 \calS_{k,A,h} \supset \I_{A', h} ,
 \]
where $A' := 2 \max\{A, k^2\}$, and so
\be \label{tmp1}
\Omega_{k,0}^\varphi (Q_n, A, t)_{\a, \b, p} \ge \sup_{0\leq h\leq t}\norm{ \wab(\cdot)
\Delta_{h\varphi(\cdot)}^k (Q_n,\cdot)}{\Lp(\I_{A', h})} .
\ee
Now it follows from \ineq{ident} that $\Delta^k_{h\varphi(x)}(Q_n, x)$ is a polynomial from $\Pn$ if $k$ is even, and it is a polynomial from $\Poly_{n-1}$ multiplied by $\varphi$ if $k$ is odd.

Hence, \ineq{remez} implies that, for $h \le 1/(\sqrt{2A'} n)$,
\begin{align} \label{tmp2}
\norm{ \wab \Delta^k_{h\varphi }(Q_n, \cdot )}{\Lp(\I_{A', h})} & \ge
\norm{ \wab \Delta^k_{h\varphi }(Q_n, \cdot )}{\Lp[-1+n^{-2}/2,1-n^{-2}/2] } \\ \nonumber
& \ge c \norm{ \wab \Delta^k_{h\varphi }(Q_n, \cdot )}{p}.
\end{align}
It now follows from \ineq{ident} that
\[
\Delta^k_{h\varphi(x)}(Q_n, x) - \varphi^{k}(x) Q_n^{(k)}(x) h^{k} =
 \sum_{i=1}^{K} \frac{1}{(2i)!}    \varphi^{k+2i}(x) Q_n^{(k+2i)}(x) h^{k+2i} \xi_{k+2i}^{2i} ,
\]
and so, as above,
\[
\norm{\wab \left( \Delta^k_{h\varphi }(Q_n, \cdot ) - \varphi^{k}  Q_n^{(k)}  h^{k}\right)}{p}^q \le 1/2 \cdot h^{kq} \norm{\wab \varphi^{k}  Q_n^{(k)}}{p}^q .
\]
Therefore,
\[
\norm{\wab  \Delta^k_{h\varphi }(Q_n, \cdot )}{p}^q \ge 1/2 \cdot h^{kq} \norm{\wab \varphi^{k}  Q_n^{(k)}}{p}^q ,
\]
which combined with \ineq{tmp1} and \ineq{tmp2} implies \ineq{2}.
\end{proof}

\sect{The polynomials of best approximation}

For   $f\in L^{\a,\b}_p$,   let $P^*_n = P^*_n(f)  \in\Pn$ and $E_n(f)_{\wab,p}$ be a polynomial and the degree of its best weighted approximation, respectively, \ie
\[
E_n(f)_{\wab,p}:= \inf_{p_n\in\Pn} \|\wab(f-p_n)\|_{p}   =   \|\wab(f-P^*_n)\|_{p} .
\]

Recall (see \cite{sam}*{Lemma 2.4} and \cite{whit}*{Theorem 1.4}) that, if $\a\ge 0$ and $\b\ge 0$, then, for any $k\in\N$, $0<p\le \infty$ and $f\in \L_p^{\a,\b}$,
\be \label{normest}
\w_{k,0}^\varphi(f, t)_{\a,\b, p} \le c  \norm{\wab f}{p}, \quad t>0 ,
\ee
with $c$ depending only on $k$, $\a$, $\b$ and $q$.
Also, for any $0<\vartheta \le 1$,
\be \label{jackest}
E_n(f)_{\wab, p} \le c  \w_{k,0}^\varphi(f, \vartheta n^{-1})_{\a,\b, p}, \quad n\geq k ,
\ee
 where $c$ depends on $\vartheta$ as well as $k$, $\a$, $\b$ and $q$.  

\begin{theorem} \label{bestapthm}
Let $k\in\N$, $\a,\b \ge 0$,   $0<p\le\infty$ and $f\in L^{\a,\b}_p$. Then, for any $n\in\N$,
\be\label{dir}
n^{-k}\|\wab  \varphi^k P_n^{*(k)}\|_{p}\le c
\w^\varphi_{k,0}(P^*_n,t)_{\a,\b,p}\le c\w^\varphi_{k,0}(f,t)_{\a,\b,p}, \quad t \ge \ccc n^{-1},
\ee
where constants  $c$ depend only on $k$, $\a$, $\b$ and  $q$.

Conversely, for $0<t \le\ccc/k$ and $n:=\lfloor\ccc/t\rfloor$,
\begin{align}\label{inv}
\w^\varphi_{k,0}(f,t)_{\a,\b,p} &\le c\left(\sum_{j=0}^\infty \w^\varphi_{k,0}(P^*_{2^jn},\ccc2^{-j}n^{-1})^q_{\a,\b,p}\right)^{1/q}\\
&\le c\left(\sum_{j=0}^\infty 2^{-jkq}n^{-kq}\|\wab\varphi^k P^{*(k)}_{2^jn}\|^q_p\right)^{1/q},\nonumber
\end{align}
where   $c$ depends  only on $k$, $\a$, $\b$  and $q$.
\end{theorem}

\begin{corollary}
Let $k\in\N$, $\a,\b \ge 0$,   $0<p\le\infty$, $f\in L^{\a,\b}_p$ and $\gamma >0$. Then,
\be\label{gamma}
\|\wab  \varphi^k P_n^{*(k)}\|_{p} = O(n^{k-\gamma}) \quad \text{iff} \quad \w^\varphi_{k,0}(f,t)_{\a,\b,p} = O(t^\gamma).
\ee
\end{corollary}

\begin{proof}[Proof of $\thm{bestapthm}$] In order to prove \ineq{dir},  one may assume that $n\ge k$.
By \thm{mainthm} we have
\begin{align*}
n^{-k}\|\wab \varphi^k  P_n^{*(k)}\|_{p}   
&\le c\ccc^{-k}\omega^\varphi_{k,0}(P^*_n,\ccc n^{-1})_{\a,\b,p}\le c\omega^\varphi_{k,0}(P^*_n,t)_{\a,\b,p}.
\end{align*}
At the same time, by \ineq{normest} and \ineq{jackest} with $\vartheta = \ccc$,
\begin{align*}
\omega^\varphi_{k,0}(P^*_n,t)^q_{\a,\b,p} & \le\omega^\varphi_{k,0}(f-P^*_n,t)^q_{\a,\b,p}+\w^\varphi_{k,0}(f,t)^q_{\a,\b,p} \\
& \le c \|\wab(f-P_n^*)\|_{p}^q + \w^\varphi_{k,0}(f,t)^q_{\a,\b,p} \\
& \le c \w^\varphi_{k,0}(f, \ccc n^{-1})^q_{\a,\b,p} +   \w^\varphi_{k,0}(f,t)^q_{\a,\b,p} \\
& \le c \w^\varphi_{k,0}(f,t)^q_{\a,\b,p} ,
\end{align*}
and  \ineq{dir} follows.

In order to prove \ineq{inv} we follow \cite{hl}. Assume that $0<t\le\ccc/k$ and note that  $n = \lfloor\ccc/t \rfloor \ge k$.
Let $\hat P_n \in \Pn$ be a polynomial of best weighted approximation of $P^*_{2n}$, \ie
\[
I_n:=\norm{\wab (P^*_{2n}-\hat P_n)}{p} = E_n(P^*_{2n})_{\wab,p} .
\]
Then,
 \ineq{jackest} with $\vartheta=\ccc/2$  implies that
\[
I_n \le c\omega^\varphi_{k,0}(P^*_{2n},\ccc (2n)^{-1})_{\a,\b,p} ,
\]
while
\[
I^q_n  \ge \|\wab(f-\hat P_n)\|_{p}^q -\norm{\wab(f-P^*_{2n})}{p}^q  \ge E_n(f)_{\wab,p}^q -E_{2n}(f)_{\wab,p}^q .
\]
Combining the above inequalities   we obtain
\begin{align*}
E_n(f)_{\wab,p}^q & =\sum_{j=0}^\infty\bigl(E_{2^jn}(f)_{\wab,p}^q -E_{2^{j+1}n}(f)_{\wab,p}^q \bigr)\le\sum_{j=0}^\infty I^q_{2^jn} \\
& \le c\sum_{j=1}^\infty\omega^\varphi_{k,0}(P^*_{2^jn},\ccc  2^{-j}n^{-1})^q_{\a,\b,p} .
\end{align*}
Hence,
\begin{align*}
\omega^\varphi_{k,0}(f,t)^q_{\a,\b,p}
&\le    c \omega^\varphi_{k,0}(f-P^*_n,t)^q_{\a,\b,p}+ c \omega^\varphi_{k,0}(P^*_n,t)^q_{\a,\b,p} \\
&\le c E_n(f)_{\wab,p}^q + c \w^\varphi_{k,0}(P^*_n,\ccc n^{-1})^q_{\a,\b,p} \\
&\le c\sum_{j=0}^\infty \w^\varphi_{k,0}(P^*_{2^{j} n},\ccc2^{-j}n^{-1})^q_{\a,\b,p}\\
&\le c\sum_{j=0}^\infty2^{-jkq}n^{-kq}\|\wab\varphi^k P^{*(k)}_{2^jn}\|^q_p,
\end{align*}
where, for the last inequality, we used \thm{mainthm}.
This completes the proof of \ineq{inv}.
\end{proof}

\sect{Further properties of the moduli}

Following \cite{sam}*{Definition 1.4}, for $k\in\N$, $r\in\N_0$ and $f\in\B^r_p(w_{\a,\b})$, $1\le  p\le\infty$, we define the weighted $K$-functional as follows
\begin{align*}
\lefteqn{ K^\varphi_{k,r}(f^{(r)},t^k)_{\a,\b,p} } \\
&\quad :=\inf_{g\in\B^{k+r}_p(\wab)}   \left\{ \norm{\wab \varphi^r (f^{(r)}-g^{(r)})}{p}
+t^k \norm{\wab\varphi^{k+r}g^{(k+r)}}{p} \right\} .
\end{align*}
We note that
\[
K_{k,\varphi}(f,t^k)_{\wab,p} = K^\varphi_{k,0}(f,t^k)_{\a,\b,p}  ,
\]
where $K_{k,\varphi}(f,t^k)_{w,p}$ is
the weighted $K$-functional that was defined in \cite{dt}*{p. 55 (6.1.1)} as
\[
K_{k,\varphi}(f,t^k)_{w,p}:=\inf_{g\in\B^k_p(w)}    \{\|w(f-g)\|_{p}+t^k\|w\varphi^k g^{(k)}\|_{p}\} .
\]

The following lemma immediately  follows from  \cite{sam}*{Corollary 1.7}.
\begin{lemma} \label{lemsam}
If $k\in\N$, $r\in\N_0$, $r/2+\a\geq 0$, $r/2+\b\geq 0$, $1\leq p\leq \infty$ and $f\in\B_p^r(\wab)$,  then,
 for all $0<t\leq 2/k$,
\[
  K^\varphi_{k,r}(f^{(r)},t^k)_{\a,\b,p}
  \leq c \wkr(f^{(r)},t)_{\a,\b,p}
    \leq c  K^\varphi_{k,r}(f^{(r)},t^k)_{\a,\b,p}   .
\]
\end{lemma}

Hence,
\be \label{dteq}
\wkr(f^{(r)},t)_{\a,\b,p} \sim K^\varphi_{k,r}(f^{(r)},t^k)_{\a,\b,p} =   K_{k,\varphi}(f^{(r)},t^k)_{w_{\a+r/2, \b+r/2},p} ,
\ee
provided that all conditions in \lem{lemsam} are satisfied.

The following sharp Marchaud inequality was proved in \cite{dd} for $f\in L^{\a,\b}_p$, $1<p<\infty$.

\begin{theorem}[\cite{dd}*{Theorem 7.5}]
For $m \in\N$, $1<p<\infty$ and $\a,\b\in J_p$,  we have
\[
K_{m,\varphi}(f,t^m)_{\wab,p} \leq C t^m \left( \int_t^1 \frac{K_{m+1,\varphi}(f,u^{m+1})_{\wab,p}^{s_*}}{u^{m{s_*}+1}}\, du + E_m(f)_{\wab,p}^{s_*} \right)^{1/{s_*}}
\]
and
\[
K_{m,\varphi}(f,t^m)_{\wab,p} \leq C t^m \left( \sum_{n <1/t} n^{{s_*}m-1} E_n(f)_{\wab,p}^{s_*} \right)^{1/{s_*}} ,
\]
where $s_* = \min\{2,p\}$.
\end{theorem}

In view of \ineq{dteq}, the following result holds.

\begin{corollary}
For   $1<p<\infty$, $r\in\N_0$, $m \in\N$, $r/2+\a\geq 0$, $r/2+\b\ge 0$ and $f\in\B^r_p(\wab)$, we have
\[
\omega^{\varphi}_{m,r} (f^{(r)},t)_{\a,\b,p} \leq C t^m \left( \int_t^1\frac{\omega^{\varphi}_{m+1,r}(f^{(r)},u)_{\a,\b, p}^{s_*}}{u^{m{s_*}+1}}\, du + E_m(f^{(r)})_{\wab\varphi^r,p}^{s_*} \right)^{1/{s_*}}
\]
and
\[
\omega^{\varphi}_{m,r}(f^{(r)},t)_{\a,\b, p} \leq C t^m \left( \sum_{n <1/t} n^{{s_*}m-1} E_n(f^{(r)})_{\wab \varphi^r,p}^{s_*} \right)^{1/{s_*}} ,
\]
where  $s_* = \min\{2,p\}$.
\end{corollary}

The following sharp Jackson inequality was proved in \cite{ddt}.

\begin{theorem}[\mbox{\cite{ddt}*{Theorem 6.2}}]
For $1<p<\infty$, $\a, \b \in J_p$ and $m \in\N$, we have
\[
2^{-nm} \left( \sum_{j=j_0}^n 2^{mj{s^*}} E_{2^j}(f)^{s^*}_{\wab,p} \right)^{1/{s^*}} \leq C K_{m,\varphi}(f, 2^{-nm})_{\wab,p}
\]
and
\[
2^{-nm} \left( \sum_{j=j_0}^n 2^{mj{s^*}}  K_{m+1,\varphi}(f, 2^{-j(m+1)})_{\wab,p}^{s^*}    \right)^{1/{s^*}} \leq C K_{m,\varphi}(f, 2^{-nm})_{\wab,p} ,
\]
where $2^{j_0} \geq m$ and $s^* = \max\{p,2\}$.
\end{theorem}

Again, by virtue of \ineq{dteq}, we have,
\begin{corollary}
For   $1<p<\infty$, $r\in\N_0$, $m \in\N$, $r/2+\a\ge 0$, $r/2+\b\ge 0$ and $f\in\B^r_p(\wab)$, we have
\[
2^{-nm} \left( \sum_{j=j_0}^n 2^{mj{s^*}} E_{2^j}(f^{(r)})^{s^*}_{\wab\varphi^r,p} \right)^{1/{s^*}} \leq C  \omega^{\varphi}_{m,r}(f^{(r)},2^{-n})_{\a,\b, p}
\]
and
\[
2^{-nm} \left( \sum_{j=j_0}^n 2^{mj{s^*}}  \omega^{\varphi}_{m+1,r}(f^{(r)},2^{-j})_{\a,\b, p}^{s^*}  \right)^{1/{s^*}} \leq C \omega^{\varphi}_{m,r}(f^{(r)},2^{-n})_{\a,\b, p},
\]
where $2^{j_0} \geq m$ and $s^* = \max\{p,2\}$.
\end{corollary}

\begin{corollary}
For   $1<p<\infty$, $r\in\N_0$, $m \in\N$, $r/2+\a\ge 0$, $r/2+\b\ge 0$ and $f\in\B^r_p(\wab)$, we have
\[
t^m  \left(  \int_t^{1/m}        \frac{\omega^{\varphi}_{m+1,r}(f^{(r)},u)_{\a,\b,p}^{s^*}}{u^{m{s^*}+1}} \, du \right)^{1/{s^*}} \leq C \omega^{\varphi}_{m,r}(f^{(r)},t)_{\a,\b,p}, \quad 0<t\leq 1/m,
\]
where  $s^* = \max\{p,2\}$.
\end{corollary}

\end{document}